\documentclass[12pt]{amsart}
 
\pdfoutput=1
\usepackage[margin=1in]{geometry}
\usepackage{amsmath,amsthm,amssymb,amsrefs,bbm,color,esint,esvect,float,graphicx,mathrsfs}

\usepackage{euscript,latexsym}
\usepackage{accents}

\usepackage{epstopdf}
\AppendGraphicsExtensions{.tif}

\newcommand{\R}{\mathbb{R}}

 \newcommand{\BB}{\mathcal{B}}   \newcommand{\DD}{\mathcal{D}}

 \newtheorem{thm}{Theorem}[section]
 \newtheorem{lemma}[thm]{Lemma}
 \newtheorem{cor}[thm]{Corollary}
 \newtheorem{prop}[thm]{Proposition}
 \newtheorem{rem}[thm]{Remark}
 \newtheorem{defin}[thm]{Definition}
 
 \numberwithin{equation}{section}
 
\newenvironment{theorem}[2][Theorem]{\begin{trivlist}
\item[\hskip \labelsep {\bfseries #1}\hskip \labelsep {\bfseries #2.}]}{\end{trivlist}}

\theoremstyle{definition}

\begin{document}

\title[On the $T1$ theorem for compactness of Calder\'on-Zygmund operators]{On the $T1$ theorem for compactness of Calder\'on-Zygmund operators}

\author{Mishko Mitkovski}
\address{Mishko Mitkovski\hfill\break\indent 
 School of Mathematical Sciences and Statistics\hfill\break\indent 
 Clemson University\hfill\break\indent 
 105 Sikes Hall\hfill\break\indent 
 Clemson, SC 29634 USA}
\email{mmitkov@clemson.edu}

\author{Cody B. Stockdale}
\address{Cody B. Stockdale\hfill\break\indent 
 School of Mathematical Sciences and Statistics\hfill\break\indent 
 Clemson University\hfill\break\indent 
 105 Sikes Hall\hfill\break\indent 
 Clemson, SC 29634 USA}
\email{cbstock@clemson.edu}

\begin{abstract}
We give a new formulation of the $T1$ theorem for compactness of Calder\'on-Zygmund singular integral operators. In particular, we prove that a Calder\'on-Zygmund operator $T$ is compact on $L^2(\mathbb{R}^n)$ if and only if $T1,T^*1\in \text{CMO}(\mathbb{R}^n)$ and $T$ is weakly compact. Compared to existing compactness criteria, our characterization more closely resembles David and Journ\'e's classical $T1$ theorem for boundedness, avoids technical conditions involving the Calder\'on-Zygmund kernel, and follows from a simpler argument. 
\end{abstract}

\maketitle


\section{Introduction}\label{introduction}
Let $T$ be a singular integral operator formally given by 
$$
    Tf(x)=\int_{\mathbb{R}^n}K(x,y)f(y)\,dy,
$$
where $K$ 
is a kernel function defined away from the diagonal $\{(x,x):x\in\mathbb{R}^n\}$ satisfying the following conditions: there exist $C_K,\delta>0$ such that
\begin{align}\label{kernelsize}
|K(x,y)|\leq \frac{C_K}{|x-y|^n}
\end{align}
whenever $x \neq y$ and
\begin{align}\label{kernelsmoothness}
    |K(x,y)-K(x,y')|, |K(y,x)-K(y',x)|\leq C_K\frac{|y-y'|^{\delta}}{|x-y|^{n+\delta}}
\end{align}
whenever $|y-y'|\leq \frac{1}{2}|x-y|$. Such kernels $K$ 
are called Calder\'on-Zygmund kernels.

The class of singular integral operators associated to Calder\'on-Zygmund kernels that act boundedly on $L^2(\mathbb{R}^n)$ was first characterized by David and Journ\'e in \cite{DJ1984} with their celebrated $T1$ theorem, which is stated as follows. 
\begin{theorem}{A}\emph{
Let $T$ be a singular integral operator associated to a Calder\'on-Zygmund kernel. Then $T$ is bounded on $L^2(\mathbb{R}^n)$ if and only if $T1,T^*1 \in \text{BMO}(\mathbb{R}^n)$ and $T$ is weakly bounded.}
\end{theorem}
\noindent Singular integral operators satisfying the conditions of Theorem A are called Calder\'on-Zygmund operators.

Much more recently, Villarroya characterized the Calder\'on-Zygmund operators that are compact on $L^2(\mathbb{R}^n)$ in \cites{V2015}; there have since been various extensions of this result in \cites{HL2021, OV2017, PPV2017, SVW2022, V2019, CYY2023,CCLLYZ2023}. The compactness of singular integral operators has proven useful in partial differential equations and geometric measure theory. For example, compactness of layer potentials was applied to solve Dirichlet and Neumann problems for Laplace's equation in bounded domains with $C^1$ boundary in \cite{FJR1978}, and compactness of layer potentials and commutators were applied to characterize regular Semmes-Kenig-Toro domains in \cite{HMT2010}.

The $T1$-type theorem for compactness of \cite{V2015} is stated as follows.
\begin{theorem}{B}\emph{
Let $T$ be a Calder\'on-Zygmund operator with kernel $K$. Then $T$ is compact on $L^2(\mathbb{R}^n)$ if and only if $T1,T^*1\in \text{CMO}(\mathbb{R}^n)$, $T$ is weakly compact, and $K$ is a compact Calder\'on-Zygmund kernel.
}
\end{theorem}
We refer to \cite{V2015} for the definition compact Calder\'on-Zygmund kernels, and to Section \ref{preliminaries} for definitions of $\text{CMO}(\mathbb{R}^n)$ and weakly compact operators. 

We present a new characterization of the compact Calder\'on-Zygmund operators. 
\begin{thm}\label{CZOCompactness}
A Calder\'on-Zygmund operator $T$ is compact on $L^2(\mathbb{R}^n)$ if and only if  \\ $T1,T^*1 \in \text{CMO}(\mathbb{R}^n)$ and $T$ is weakly compact.
\end{thm}


While our result and Villaroya's Theorem B are not comparable in the sense that neither is a consequence of the other, our Theorem \ref{CZOCompactness} has the advantage that one does not need to check the additional conditions on the kernel in order to deduce the compactness of $T$. In addition, our proof is significantly shorter and avoids many technical intricacies of the proof of Theorem B in \cite{V2015}. 

Combining Theorem \ref{CZOCompactness} with 
the recent theorems of \cites{HL2021, SVW2022}, one immediately 
obtains the following compactness result for Calder\'on-Zygmund operators on weighted Lebesgue spaces with respect to Muckenhoupt $A_p$ weights. 
\begin{cor}\label{WeightedCorollary}
If $T$ is a weakly compact Calder\'on-Zygmund operator such that $T1,T^*1\in\text{CMO}(\mathbb{R}^n)$, then $T$ is compact on $L^p(\omega)$ for all $p \in (1,\infty)$ and all $\omega \in A_p$. 
\end{cor}

The paper is organized as follows. In Section \ref{preliminaries}, we introduce the setting and discuss our strategy. In Section \ref{SectionLocalization}, we characterize the compact operators within a general class of localized operators. In Section \ref{T1Section}, we prove Theorem \ref{CZOCompactness} by showing that cancellative Calder\'on-Zygmund operators 
are localized and by treating appropriate paraproduct terms. 


\section{Preliminaries}\label{preliminaries}

We write $A\lesssim B$ if $A\leq CB$ for some constant $C>0$, and write $A\approx B$ if $A\lesssim B$ and $B \lesssim A$. For $x \in \mathbb{R}^n$ and $r>0$, we set $B(x,r)$ to be the Euclidean ball centered at $x$ of radius $r$. We use $|A|$ to denote the Lebesgue measure of a set $A\subseteq \mathbb{R}^n$. 

Let $\DD(\R^n)$ be the 
space of infinitely differentiable and compactly supported functions on $\R^n$ and $\DD'(\R^n)$ be its dual, the space of distributions. Let $T: \DD(\R^n)\to \DD'(\R^n)$ be a linear operator that is continuous in the sense of distributions and suppose that 
the Schwartz distributional kernel of $T$ coincides with a Calder\'on-Zygmund kernel $K$ in the sense that
$$
    \langle Tf,g\rangle =\int_{\mathbb{R}^n}\int_{\mathbb{R}^n}K(x,y)f(y)g(x)\,dxdy
$$
for some $K:\mathbb{R}^n\times\mathbb{R}^n\setminus\{(x,x):x\in\mathbb{R}^n\}\rightarrow\mathbb{C}$ satisfying \eqref{kernelsize} and \eqref{kernelsmoothness}, and all $f, g \in \DD(\mathbb{R}^n)$ with disjoint supports. The adjoint operator $T^*$ is defined by 
$$
    \langle Tf,g\rangle = \langle f,T^*g\rangle
$$
for $f,g \in \DD(\mathbb{R}^n)$ with disjoint supports. Note that $T^*$ is also a singular integral operator associated to a Calder\'on-Zygmund kernel given by $\widetilde{K}(x,y):=K(y,x)$. 
As mentioned in the introduction, Theorem A provides necessary and sufficient conditions for such operators to extend boundedly on $L^2(\mathbb{R}^n)$. 

The most common strategy used in proving Theorem A involves decomposing $T$ as 
\begin{align*}
    T=S+P_{T1}+P_{T^*1}^*,
\end{align*}
where $S$ is a singular integral operator associated to a Calder\'on-Zygmund kernel satisfying $S1=S^*1=0$, and $P_{T1}$ and $P_{T^*1}^*$ are paraproduct operators. One proves that $S$ is bounded using the weak boundedness assumption, whereas $P_{T1}$ and $P_{T^*1}^*$ are handled by an argument relying on the intimate connection between $\text{BMO}(\mathbb{R}^n)$ and Carleson measures. This strategy can be understood more clearly when the singular integral operator $T$ is represented in a wavelet basis -- from this perspective, the operator $S$ is realized as an almost-diagonal term, and the two paraproducts are viewed as off-diagonal parts of $T$. This point of view, which we here adopt, was promoted by Coifman and Meyer in \cite{CoifmanMeyer}.

Since compactly supported continuous wavelet frames are much easier to construct, we choose to use a continuous frame of wavelets to represent Calder\'on-Zygmund operators, rather than a discrete wavelet basis, in order to further simplify our argument. As usual, the continuous frames of wavelets are indexed by the $ax+b$ group, which is a natural identification of Euclidean balls (or cubes) in $\mathbb{R}^n$ with points in $\mathbb{R}^{n+1}_+:=(0,\infty)\times \mathbb{R}^n$. The $ax+b$ group is the group $(\mathbb{R}^{n+1}_+,*)$ with operation given by
$$
    (a,b)*(a',b'):=(aa',ab'+b)
$$
for $a>0$ and $b \in \mathbb{R}^n$. The identity element of $\mathbb{R}^{n+1}_+$ is $(1,0)$, and inverses are given by $(a,b)^{-1}=\left(\frac{1}{a},-\frac{b}{a}\right)$. We denote the left-Haar measure by $\lambda$; recall that $\lambda$ is given by
$$
    d\lambda(a,b)=\frac{dadb}{a^{n+1}}.
$$ 
The left-invariant metric, d, is the Riemannian metric given by the length element 
$$
    ds^2=\frac{da^2+db^2}{a^2}.
$$
Set 
$$
    D((a,b),r):=\{(a',b')\in\mathbb{R}^{n+1}_+:d((a,b),(a',b'))<r\}
$$
and $D((a,b),r)^c:=\mathbb{R}^{n+1}_+\setminus D((a,b),r).$ We emphasize that $D$ denotes balls in $\mathbb{R}^{n+1}_+$ with respect to the hyperbolic metric $d$ above, and $B$ denotes Euclidean balls in $\R^n$.

Continuous wavelet frames are generated by the unitary representation $U$ of the $ax+b$ group on $L^2(\mathbb{R}^n)$ given by 
$$
    f_{(a,b)}:=U(a,b)f=\frac{1}{a^{n/2}}f\left(\frac{\cdot -b}{a}\right).
$$
Note that $U(a,b)U(a',b')=U((a,b)*(a',b'))$ and that $U(a,b)^*=U(a,b)^{-1}=U(a,b)^{-1}$. If $T$ is a singular integral operator associated to a Calder\'on-Zygmund kernel $K$, then 
$$
    T_{(a,b)}:=U(a,b)^{*}TU(a,b)
$$ 
is a singular integral operator associated to the Calder\'on-Zygmund kernel 
$$
     K_{(a,b)}(x,y):=a^nK(ax+b,ay+b)
$$ 
which also satisfies \eqref{kernelsize} and \eqref{kernelsmoothness} with the same constants as $T$. Thus the $ax+b$ group can be viewed as the group of translations and dilations in $\mathbb{R}^n$, and Calder\'on-Zygmund constants of singular integral operators are invariant under their conjugation. 

Let $\psi \in \mathcal{D}(\mathbb{R}^n)$ be a function (mother wavelet) such that  
\[
    \text{supp}\,\psi \subseteq B(0,1), \quad  \int_{\mathbb{R}^n} \psi(x)\,dx =0, \quad  \text{and}\quad
\int_0^{\infty}\frac{|\widehat{\psi}(x\xi)|}{\xi}\,d\xi=1.
\]
It is well-known that under the above conditions, $\{\psi_{(a,b)}\}_{(a,b)\in\mathbb{R}^{n+1}_+}$ forms a continuous Parseval frame for $L^2(\mathbb{R}^n)$ in the sense that
$$
    f=\int_{\mathbb{R}^{n+1}_+}\langle f,\psi_{(a,b)}\rangle\psi_{(a,b)}\,d\lambda(a,b)
$$
and
$$
    \|f\|_{L^2(\mathbb{R}^n)}^2=\int_{\mathbb{R}^{n+1}_+}|\langle f,\psi_{(a,b)}\rangle|^2\,d\lambda(a,b)
$$
for any $f \in L^2(\mathbb{R}^n)$, see \cite{D1990}.  Here and everywhere below we use $\langle \cdot,\cdot\rangle$ to denote the usual inner product on $L^2(\mathbb{R}^n)$. The first of these two equivalent identities is the classical Calder\'on reproducing formula. 

In \cite{DJ1984}, David and Journ\'e define $T$ to be weakly bounded if for every bounded subset $\BB \subseteq \mathcal{D}(\mathbb{R}^n)$, there exists $C>0$ such that 
$$
    |\langle Tf_{(a,b)}, g_{(a,b)}\rangle|\leq C
$$
for any $f,g \in \BB$ and any $(a,b) \in \mathbb{R}^{n+1}_+$. While this property is clearly necessary for $T$ to be bounded on $L^2(\mathbb{R}^n)$, it is also sufficient for the boundedness of the almost-diagonal term $S$. We use a similar \textit{weak compactness property} to obtain the compactness of $S$ in the proof of Theorem \ref{CZOCompactness}. 
\begin{defin}\label{weakcompdef}
Let $T$ be bounded on $L^2(\mathbb{R}^n)$. Then $T$ is weakly compact if for every bounded subset $\BB \subseteq \mathcal{D}(\mathbb{R}^n)$ and every $f\in \BB$, we have
$$
    \lim_{(a,b)\to \infty} \sup_{g\in \BB} |\langle Tf_{(a,b)}, g_{(a,b)}\rangle|=0.
$$

\end{defin}
\noindent Here and below we write $\lim_{(a,b)\rightarrow \infty}F(a,b)$ to denote $\lim_{R\rightarrow \infty}\sup_{(a,b) \in D((1,0),R)^c}F(a,b)$. Note that the weak compactness property is necessary for $T$ to be compact on $L^2(\mathbb{R}^n)$. Indeed, since $f_{(a,b)}\to 0$ weakly as $(a,b)\to \infty$, we have that $Tf_{(a,b)}\to 0$ strongly for compact $T$, so 
$$
    \sup_{g\in \BB}|\langle Tf_{(a,b)}, g_{(a,b)}\rangle|\leq \|Tf_{(a,b)}\|_{L^2(\mathbb{R}^n)}\sup_{g\in \BB}\|g\|_{L^2(\mathbb{R}^n)}\to 0
$$
as $(a,b)\to \infty$. 
 
While the almost-diagonalized term $S$ is dealt with using the weak boundedness condition in the proof of Theorem A, the paraproducts are handled with the $\text{BMO}(\mathbb{R}^n)$ hypotheses on the symbols $T1$ and $T^*1$ and the connection between $\text{BMO}(\mathbb{R}^n)$ and Carleson measures via the frame $\{\psi_{(a,b)}\}_{(a,b)\in\mathbb{R}^{n+1}_+}$. 
Define the Carleson tent over $B=B(x,r)\subseteq\mathbb{R}^n$ by
$$
    T(B):=\{(a,b) \in \mathbb{R}^{n+1}_+: |x-b|<r-a\}
$$ 
and define the cone over $x \in \mathbb{R}^n$ by
$$
    V_x:=\{(a,b)\in\mathbb{R}^{n+1}_+: |x-b|<a\}.
$$
For a Borel measure $\mu$ on $\mathbb{R}^{n+1}_+$, define
$$
    C\mu(x):=\sup_{(a,b)\in V_x}\frac{\mu(T(B(b,a)))}{|B(b,a)|}.
$$
Recall that a measure $\mu$ on $\mathbb{R}^{n+1}_+$ is a Carleson measure if $\|C\mu\|_{L^{\infty}(\mathbb{R}^n)}<\infty$. 
The well-known connection between $\text{BMO}(\mathbb{R}^n)$ and Carleson measures is given as follows: a locally integrable function
$f$ is in $\text{BMO}(\mathbb{R}^n)$ if and only if 
$$
    d\mu_f(a,b):=|\langle f,\psi_{(a,b)}\rangle|^2\,d\lambda(a,b)
$$ 
defines a Carleson measure on $\mathbb{R}^{n+1}_+$. 

To address the compactness of paraproducts, the role of $\text{BMO}(\R^n)$ is played by its subspace $\text{CMO}(\mathbb{R}^n)$.
Coifman and Weiss originally defined $\text{CMO}(\mathbb{R}^n)$ in \cite{CW1977}*{p. 638} as the closure of $\mathcal{D}(\mathbb{R}^n)$ with respect to the $\text{BMO}(\mathbb{R}^n)$ norm -- we remark that there are several equivalent formulations of $\text{CMO}(\mathbb{R}^n)$, including a connection with vanishing Carleson measures, see \cites{DM2016,N1977,U1978}.  
We will use the following one: A function $f$ is in $\text{CMO}(\mathbb{R}^n)$ if $f \in \text{BMO}(\mathbb{R}^n)$ and $\mu_f$ is a vanishing Carleson measure.
Recall that a Carleson measure $\mu$ on $\R^{n+1}_+$ is a vanishing Carleson measure if 
$$
    \lim_{(a,b)\rightarrow\infty} \frac{\mu(T(B(b,a)))}{|B(b,a)|}=0.
$$

Given a radial, non-increasing, bounded, and integrable function $\varphi$ on $\mathbb{R}^n$, define the non-tangential maximal function of $f$ by 
$$
	Mf(x):=\sup_{(a,b)\in V_x} |\langle f,\varphi_{(a,b)}\rangle|.
$$
Under the assumptions on $\varphi$, $M$ is pointwise controlled by the Hardy-Littlewood maximal function, and therefore is bounded on $L^p(\mathbb{R}^n)$ for all $p \in (1,\infty)$, see \cite{S1993}*{p. 57}. We use the following property of Carleson measures which follows from the arguments of \cite{S1993}*{p. 59--61}.
\begin{lemma}\label{SteinLemma}
    If $p>0$ and $\varphi \in L^1(\mathbb{R}^n)$ is radial, non-increasing, and bounded, then
    $$
        \int_{\mathbb{R}^{n+1}_+}|\langle f,\varphi_{(a,b)}\rangle|^p\,d\mu(a,b) \lesssim \int_{\mathbb{R}^n}Mf(x)^pC\mu(x)\,dx
    $$
    for all $f:\mathbb{R}^n\rightarrow \mathbb{R}$ and all Borel measures $\mu$ on $\mathbb{R}^{n+1}_+$ for which the right-hand side is finite.
\end{lemma}




\section{Compactness characterization for localized operators}
\label{SectionLocalization}

We introduce a class of localized (or almost-diagonalized) operators and characterize the compact operators among the collection of all localized operators. We apply this abstraction to characterize the compact Calder\'on-Zygmund operators $T$ satisfying $T1=T^*1=0$.
\begin{defin}
An operator $T$ on $L^2(\mathbb{R}^n)$ is localized with respect to $\{\psi_{(a,b)}\}_{(a,b)\in\mathbb{R}^{n+1}_+}$ if 
\begin{align}\label{Schurcondition}
    &\sup_{(a',b')\in\mathbb{R}^{n+1}_+}w(a',b')^{-1}\int_{\mathbb{R}^{n+1}_+}|\langle T\psi_{(a',b')},\psi_{(a,b)}\rangle|w(a,b)\,d\lambda(a,b)<\infty,\\
    &\sup_{(a',b')\in\mathbb{R}^{n+1}_+}w(a',b')^{-1}\int_{\mathbb{R}^{n+1}_+}|\langle T^*\psi_{(a',b')},\psi_{(a,b)}\rangle|w(a,b)\,d\lambda(a,b)<\infty,\notag
\end{align}
and
\begin{align}\label{localizationcondition}
    &\lim_{R\rightarrow\infty}\sup_{(a',b')\in\mathbb{R}^{n+1}_+}w(a',b')^{-1}\int_{D((a',b'),R)^c}|\langle T\psi_{(a',b')},\psi_{(a,b)}\rangle|w(a,b)\,d\lambda(a,b)=0,\\
    &\lim_{R\rightarrow\infty}\sup_{(a',b')\in\mathbb{R}^{n+1}_+}w(a',b')^{-1}\int_{D((a',b'),R)^c}|\langle T^*\psi_{(a',b')},\psi_{(a,b)}\rangle|w(a,b)\,d\lambda(a,b)=0\notag
\end{align}
for some positive function $w$ on $\mathbb{R}^{n+1}_+$.
\end{defin}


The following is a version of the Riesz-Kolmogorov theorem in this setting, see \cites{DFG2002,MSWW2023}.
\begin{thm}\label{RieszKolmogorov}
A bounded set $\mathcal{K}\subseteq L^2(\mathbb{R}^n)$ is precompact if and only if 
$$
    \lim_{R\rightarrow\infty}\sup_{f\in\mathcal{K}}\int_{D((1,0),R)^c}|\langle f,\psi_{(a,b)}\rangle|^2\,d\lambda(a,b)=0.
$$
\end{thm}
\noindent
Applying Theorem \ref{RieszKolmogorov} with $\mathcal{K}$ equal to the image of the closed unit ball in $L^2(\mathbb{R}^n)$ under a bounded operator $T$ yields the following corollary.
\begin{cor}\label{RKcorollary}
Let $T$ be a bounded operator on $L^2(\mathbb{R}^n)$. Then $T$ is compact on $L^2(\mathbb{R}^n)$ if and only if 
$$
    \lim_{R\rightarrow\infty}\sup_{\substack{f\in L^2(\mathbb{R}^n)\\ \|f\|_{L^2(\mathbb{R}^n)}\leq1}}\int_{D((1,0),R)^c}|\langle Tf,\psi_{(a,b)}\rangle|^2\,d\lambda(a,b)=0.
$$
\end{cor}

\begin{prop}\label{compactnessproposition}
Let $T$ be a bounded operator on $L^2(\mathbb{R}^n)$ that satisfies \eqref{Schurcondition} for some positive function $w$ on $\mathbb{R}^{n+1}_+$. If 
\begin{align}\label{compactnessclaim}
    \lim_{R\rightarrow\infty}\sup_{(a',b')\in\mathbb{R}^{n+1}_+}w(a',b')^{-1}\int_{D((1,0),R)^c}|\langle T\psi_{(a',b')},\psi_{(a,b)}\rangle|w(a,b)\,d\lambda(a,b)=0,
\end{align}
then $T$ is compact on $L^2(\mathbb{R}^n)$.
\end{prop}
\begin{proof}
We verify the condition of Corollary \ref{RKcorollary}. Let $f \in L^2(\mathbb{R}^n)$ with $\|f\|_{L^2(\mathbb{R}^n)}\leq 1$. Then
\begin{align*}
    \langle Tf, \psi_{(a,b)}\rangle &=\left\langle T\left(\int_{\mathbb{R}^{n+1}_+}\langle f,\psi_{(a',b')}\rangle \psi_{(a',b')}\,d\lambda(a',b')\right), \psi_{(a,b)}\right\rangle\\
    &= \int_{\mathbb{R}^{n+1}_+}\langle f,\psi_{(a',b')}\rangle\langle T\psi_{(a',b')},\psi_{(a,b)}\rangle\,d\lambda(a',b').
\end{align*}
By Cauchy-Schwarz and the localization condition \eqref{Schurcondition}, we have
\begin{align*}
    |\langle Tf,\psi_{(a,b)}\rangle|^2&\leq \left(\int_{\mathbb{R}^{n+1}_+}|\langle f,\psi_{(a',b')}\rangle||\langle T\psi_{(a',b')},\psi_{(a,b)}\rangle|\,d\lambda(a',b')\right)^2\\
    &\leq \left(\int_{\mathbb{R}^{n+1}_+} |\langle T\psi_{(a',b')},\psi_{(a,b)}\rangle|\frac{w(a',b')}{w(a,b)}\,d\lambda(a',b')\right)\\
    &\quad\quad\times\left(\int_{\mathbb{R}^{n+1}_+}|\langle T\psi_{(a',b')},\psi_{(a,b)}\rangle||\langle f,\psi_{(a',b')}\rangle|^2\frac{w(a,b)}{w(a',b')}\,d\lambda(a',b')\right)\\
    &\leq C\int_{\mathbb{R}^{n+1}_+}|\langle T\psi_{(a',b')},\psi_{(a,b)}\rangle||\langle f,\psi_{(a',b')}\rangle|^2\frac{w(a,b)}{w(a',b')}\,d\lambda(a',b'),
\end{align*}
where $C$ is the finite constant from \eqref{Schurcondition}. Therefore, for any $R>0$, we have 
\begin{align*}
    \int_{D((1,0),R)^c} &|\langle Tf,\psi_{(a,b)}\rangle|^2\,d\lambda(a,b) \\
    &\leq C\int_{D((1,0),R)^c}\int_{\mathbb{R}^{n+1}_+}|\langle T\psi_{(a',b')},\psi_{(a,b)}\rangle||\langle f,\psi_{(a',b')}\rangle|^2\frac{w(a,b)}{w(a',b')}\,d\lambda(a',b')d\lambda(a,b)\\
    &= C \int_{\mathbb{R}^{n+1}_+}|\langle f,\psi_{(a',b')}\rangle|^2\int_{D((1,0),R)^c} |\langle T\psi_{(a',b')},\psi_{(a,b)}\rangle|\frac{w(a,b)}{w(a',b')}\,d\lambda(a,b)d\lambda(a',b').
\end{align*}

Let $\varepsilon>0$ be given. Apply the hypothesis \eqref{compactnessclaim} to get a constant $R>0$ such that $\int_{D((1,0),R)^c} |\langle T\psi_{(a',b')},\psi_{(a,b)}\rangle|\frac{w(a,b)}{w(a',b')}\,d\lambda(a,b)< \frac{\varepsilon}{C}$ for all $(a',b')\in\mathbb{R}^{n+1}_+$. Then 
$$
    \int_{D((1,0),R)^c} |\langle Tf,\psi_{(a,b)}\rangle|^2\,d\lambda(a,b) < C\frac{\varepsilon}{C}\int_{\mathbb{R}^{n+1}_+}|\langle f,\psi_{(a',b')}\rangle|^2\,d\lambda(a',b') \leq \varepsilon.
$$
Therefore $T$ is compact by Corollary \ref{RKcorollary}.
\end{proof}

\begin{thm}\label{Generalcompactness}
Let $T$ be bounded on $L^2(\mathbb{R}^n)$ and localized with respect to $\{\psi_{(a,b)}\}_{(a,b)\in\mathbb{R}^{n+1}_+}$. Then $T$ is compact on $L^2(\mathbb{R}^n)$ if and only if $T$ is weakly compact.
\end{thm}
\begin{proof}
The weak compactness of $T$ is necessary by the argument following Definition \ref{weakcompdef}.
We prove the reverse implication by verifying condition \eqref{compactnessclaim} and applying Proposition \ref{compactnessproposition}. Fix $(a',b')\in\mathbb{R}^{n+1}_+$ and let $\varepsilon>0$. Apply \eqref{localizationcondition} to get $R_0>0$ such that
$$
w(a',b')^{-1}\int_{D((a',b'),R_0)^c}|\langle T\psi_{(a',b')},\psi_{(a,b)}\rangle|w(a,b)\,d\lambda(a,b)<\frac{\varepsilon}{2}
$$
for all $(a',b')\in\mathbb{R}^{n+1}_+$. Let $C>0$ be a constant 
such that $\frac{w(a,b)}{w(a',b')}\leq C$ for all $(a,b) \in D((a',b'),R_0)$. 
Use the weak compactness to select $N_0>0$ so that $|\langle T\psi_{(a',b')},\psi_{(a,b)}\rangle| < \frac{\varepsilon}{2C\lambda(D((1,0),R_0))}$ whenever $(a',b') \in D((1,0),N_0)^c$. 
Set $R=R_0+N_0$. 

For arbitrary $(a',b') \in \mathbb{R}^{n+1}_+$, write
\begin{align*}
    w(a',b')^{-1}&\int_{D((1,0),R)^c} |\langle T\psi_{(a',b')},\psi_{(a,b)}\rangle|w(a,b)\,d\lambda(a,b)\\
    &= w(a',b')^{-1}\int_{D((1,0),R)^c\cap D((a',b'),R_0)^c} |\langle T\psi_{(a',b')},\psi_{(a,b)}\rangle|w(a,b)\,d\lambda(a,b)\\
    &\quad\quad+w(a',b')^{-1}\int_{D((1,0),R)^c\cap D((a',b'),R_0)} |\langle T\psi_{(a',b')},\psi_{(a,b)}\rangle|w(a,b)\,d\lambda(a,b).
\end{align*}
The first term above is controlled above by $\varepsilon/2$ due to the choice of $R_0$. 
For the second term, we only need to consider the case when $D((1,0),R)^c \cap D((a',b'),R_0) \neq \emptyset$. If $(a,b) \in D((1,0),R)^c \cap D((a',b'),R_0)$, then 
$$
    d((a',b'),(1,0))\ge d((a,b),(1,0))-d((a',b'),(a,b))\ge R- R_0 = N_0,
$$
so $(a',b') \in D((1,0),N_0)^c$. We thus control the second term above by choice of $N_0$:
\begin{align*}
    w(a',b')^{-1}&\int_{D((1,0),R)^c \cap D((a',b'),R_0)} |\langle T\psi_{(a',b')},\psi_{(a,b)}\rangle|w(a,b)\,d\lambda(a,b)\\
    &< C\lambda(D((a',b'),R_0))\frac{\varepsilon}{2C\lambda(D((1,0),R_0))}=\frac{\varepsilon}{2}.
\end{align*}
This establishes \eqref{compactnessclaim} and completes the proof.
\end{proof}


\section{$T1$ theorem for compactness of Calder\'on-Zygmund integrals}\label{T1Section}

We show that singular integral operators $T$ associated to Calder\'on-Zygmund kernels that satisfy $T1=T^*1=0$ are localized with respect to $\{\psi_{(a,b)}\}_{(a,b)\in\mathbb{R}^{n+1}_+}$.

The following standard estimate is crucial. 
\begin{lemma}\label{matrixcoefficients}
Let $T$ be a singular integral operator associated to a Calder\'on-Zygmund kernel. If $T$ is weakly bounded and $T1=T^*1=0$, then 
$$
    |\langle T\psi,\psi_{(a,b)}\rangle|\lesssim
    \begin{cases}
    \frac{1}{a^{n/2+\delta}}&a\ge 1, \,\, |b|\leq a\\
    \frac{a^{n/2}}{|b|^{n+\delta}} & a\ge 1,\,\,|b|> a\\
    a^{n/2+\delta} & a< 1, \,\, |b|\leq 1\\
    \frac{a^{n/2+\delta}}{|b|^{n+\delta}} & a< 1, \,\, |b|> 1
    \end{cases}.
$$
\end{lemma}
\begin{rem} For a high level proof of this estimate see~\cite{DT2015}*{Theorem 4.4}. More detailed proof can be found in~\cite{CoifmanMeyer}*{p. 52-53}. It is important to remark that for this estimate to hold, the mother wavelet $\psi$ needs to have at least one bounded derivative. In particular, the standard Haar wavelet cannot be used in our argument. 
\end{rem}


\begin{thm}\label{CZOlocalization}
Let $T$ be a singular integral operator associated to a Calder\'on-Zygmund kernel. If $T$ is weakly bounded and $T1=T^*1=0$, then $T$ is localized with respect to $\{\psi_{(a,b)}\}_{(a,b)\in\mathbb{R}^{n+1}_+}$.
\end{thm}

\begin{proof}
Let $w$ be given by $w(a,b)=a^{n/2}$. Using the substitution $\left(\frac{a}{a'},\frac{b-b'}{a'}\right) \mapsto (a,b)$ we obtain
\begin{align*}
    \sup_{(a',b')\in\mathbb{R}^{n+1}_+}&w(a',b')^{-1}\int_{\mathbb{R}^{n+1}_+}|\langle T\psi_{(a',b')},\psi_{(a,b)}\rangle|w(a,b)\,d\lambda(a,b)\\
    &=\sup_{(a',b')\in\mathbb{R}^{n+1}_+}w(a',b')^{-1}\int_{\mathbb{R}^{n+1}_+}\left|\left\langle T_{(a',b')}\psi,U(a',b')^*\psi_{(a,b)}\right\rangle\right|w(a,b)\,d\lambda(a,b)\\
    &=\sup_{(a',b')\in\mathbb{R}^{n+1}_+}\int_{\mathbb{R}^{n+1}_+}\left|\left\langle T_{(a',b')}\psi,\psi_{\left(\frac{a}{a'},\frac{b-b'}{a'}\right)}\right\rangle\right|w\left(\frac{a}{a'},\frac{b-b'}{a'}\right)\,d\lambda(a,b)\\
    &=\sup_{(a',b')\in\mathbb{R}^{n+1}_+}\int_{\mathbb{R}^{n+1}_+}|\langle T_{(a',b')}\psi,\psi_{(a,b)}\rangle|w(a,b)\,d\lambda(a,b).
\end{align*}
Recall that $T_{(a',b')}$ is a singular integral operator associated to a Calder\'on-Zygmund kernel satisfying \eqref{kernelsize} and \eqref{kernelsmoothness} with the same constants as $T$. Also, note that we have $T_{(a',b')}1=T_{(a',b')}^*1=0$ since $T1=T^*1=0$. Splitting $\mathbb{R}^{n+1}_+$ into four regions and applying the bounds of Lemma \ref{matrixcoefficients}, we obtain
\begin{align*}
    &\sup_{(a',b')\in\mathbb{R}^{n+1}_+}\int_{\mathbb{R}^{n+1}_+}|\langle T_{(a',b')}\psi,\psi_{(a,b)}\rangle|w(a,b)\,d\lambda(a,b)\\
    &\quad\quad\lesssim \int_{1}^{\infty}\int_{|b|\leq a}\frac{1}{a^{n+1+\delta}}\,dbda+\int_{1}^{\infty}\int_{|b|>a}\frac{1}{a|b|^{n+\delta}}\,dbda\\
    &\quad\quad\quad+\int_0^1\int_{|b|\leq 1} \frac{1}{a^{1-\delta}}\,dbda+\int_{0}^1\int_{|b|>1}\frac{1}{a^{1-\delta}|b|^{n+\delta}}\,dbda \,\,\lesssim 1.
\end{align*} 
This establishes the first localization condition of \eqref{Schurcondition}. The second condition of \eqref{Schurcondition} follows from an identical argument replacing $T$ with $T^*$.

To address \eqref{localizationcondition}, we follow the steps as in the computation above to get
\begin{align*}
    \sup_{(a',b')\in\mathbb{R}^{n+1}_+}&w(a',b')^{-1}\int_{D((a',b'),R)^c}|\langle T\psi_{(a',b')},\psi_{(a,b)}\rangle|w(a,b)\,d\lambda(a,b)\\
    &=\sup_{(a',b')\in\mathbb{R}^{n+1}_+}\int_{D((1,0),R)^c}|\langle T_{(a',b')}\psi,\psi_{(a,b)}\rangle|w(a,b)\,d\lambda(a,b)\\
    &\lesssim \iint_{\left\{(a,b)\in\mathbb{R}^{n+1}_+:\, a\ge 1, \,\, |b|\leq a\right\}\setminus D((1,0),R)}\frac{1}{a^{n+1+\delta}}\,dadb\\
    &\quad\quad+\iint_{\left\{(a,b)\in\mathbb{R}^{n+1}_+:\, a\ge 1, \,\, |b|> a\right\}\setminus D((1,0),R)}\frac{1}{a|b|^{n+\delta}}\,dadb\\
    &\quad\quad+\iint_{\left\{(a,b)\in\mathbb{R}^{n+1}_+:\, 0<a<1, \,\, |b|\leq 1\right\}\setminus D((1,0),R)} \frac{1}{a^{1-\delta}}\,dadb\\
    &\quad\quad+\iint_{\left\{(a,b)\in\mathbb{R}^{n+1}_+:\, 0<a<1, \,\, |b|>1\right\}\setminus D((1,0),R)}\frac{1}{a^{1-\delta}|b|^{n+\delta}}\,dadb
\end{align*}
for any $R>0$. Since each of the four integrals above is the tail of a convergent integral, the above expression can be made arbitrarily small by choosing $R$ sufficiently large. 
This, and an identical argument with $T^*$ in place of $T$, establishes \eqref{localizationcondition}.
\end{proof}

We next address the compactness of the paraproduct operators. Let $\varphi \in \mathcal{D}(\mathbb{R}^n)$ be radial, non-increasing, and satisfy $\varphi(x)=1$ for $x \in B(0,\frac{1}{2})$ and $\varphi(x)=0$ for $x \not \in B(0,1)$. The paraproduct operator with symbol $\beta$ is given by
$$
    P_{\beta}f:=\int_{\mathbb{R}^{n+1}_+}\langle f,\widetilde{\varphi}_{(a,b)}\rangle\langle \beta,\psi_{(a,b)}\rangle\psi_{(a,b)}\,d\lambda(a,b),
$$
where $\widetilde{\varphi}_{(a,b)}:=a^{-n}\varphi\left(\frac{\cdot - b}{a}\right)$.


\begin{thm}\label{ParaproductCompactness}
    If $\beta \in \text{CMO}(\mathbb{R}^n)$, then $P_{\beta}$ is compact on $L^2(\mathbb{R}^n)$.
\end{thm}
\begin{proof}
Let $\{f_k\}_{k=1}^{\infty}\subseteq L^2(\mathbb{R}^n)$ converge to $0$ weakly. We show that $\|P_{\beta}f_k\|_{L^2(\mathbb{R}^n)}\rightarrow 0$. Let $\varepsilon>0$. Use the assumption $\beta \in \text{CMO}(\mathbb{R}^n)$ to choose $R>0$ such that $\frac{\mu_{\beta}(T(B(b,a)))}{|B(b,a)|}<\varepsilon$ whenever $(a,b) \in D((1,0),R)^c$. Let $r>0$ be so that $B=B(0,r)$ satisfies $T(B)\supseteq D((1,0),R)$ and put $\widetilde{T}(B):= \{(a,b) \in T(B) : a>a_0\}$, where $a_0$ is chosen so that $d((1,0),(2a_0,0))=R$. Then
\begin{align*}
    \|P_{\beta}f_k&\|_{L^2(\mathbb{R}^n)}^2=\int_{\mathbb{R}^{n+1}_+}|\langle f_k,\widetilde{\varphi}_{(a,b)}\rangle|^2|\langle \beta,\psi_{(a,b)}\rangle|^2\,d\lambda(a,b)\\
    &=\int_{\widetilde{T}(B)}|\langle f_k,\widetilde{\varphi}_{(a,b)}\rangle|^2|\langle \beta,\psi_{(a,b)}\rangle|^2\,d\lambda(a,b)+\int_{\widetilde{T}(B)^c}|\langle f_n,\widetilde{\varphi}_{(a,b)}\rangle|^2|\langle \beta,\psi_{(a,b)}\rangle|^2\,d\lambda(a,b).
\end{align*}

To handle the first term, note that $f_k\rightarrow 0$ weakly implies that $\{f_k\}$ is bounded in $L^2(\mathbb{R}^n)$, and so $|\langle f_k,\widetilde{\varphi}_{(a,b)}\rangle|^2\leq a^{-n/2}\|f_k\|_{L^2(\mathbb{R}^n)}\|\varphi_{(a,b)}\|_{L^2(\mathbb{R}^n)}$ is bounded independently of $k\in\mathbb{N}$ and $(a,b) \in \widetilde{T}(B)$. Also, since $\mu_{\beta}$ is a Carleson measure, $\int_{T(B)}|\langle \beta,\psi_{(a,b)}\rangle|^2\,d\lambda(a,b)\leq|B|<\infty$. The dominated convergence theorem thus gives that 
$$
    \lim_{k\rightarrow\infty}\int_{\widetilde{T}(B)}|\langle f_k,\widetilde{\varphi}_{(a,b)}\rangle|^2|\langle \beta,\psi_{(a,b)}\rangle|^2\,d\lambda(a,b)= \int_{\widetilde{T}(B)}\lim_{k\rightarrow\infty}|\langle f_k,\widetilde{\varphi}_{(a,b)}\rangle|^2|\langle \beta,\psi_{(a,b)}\rangle|^2\,d\lambda(a,b) =0.
$$

We now address the second term. Define $\mu_{\widetilde{T}(B)^c}$ by 
$$
    d\mu_{\widetilde{T}(B)^c}(a,b)=\chi_{\widetilde{T}(B)^c}(a,b)|\langle \beta,\psi_{(a,b)}\rangle|^2\,d\lambda(a,b).
$$
By choice of $B$ and $a_0$, we have that $\|C(\mu_{\widetilde{T}(B)^c})\|_{L^{\infty}(\mathbb{R}^n)}\lesssim \varepsilon$.  By Lemma \ref{SteinLemma}, we have
\begin{align*}
	\int_{\widetilde{T}(B)^c}|\langle f_k,\widetilde{\varphi}_{(a,b)}\rangle|^2&|\langle \beta,\psi_{(a,b)}\rangle|^2\,d\lambda(a,b) \lesssim \int_{\mathbb{R}^n}Mf_k(x)^2C(\mu_{\widetilde{T}(B)^c})(x)\,dx\\
&\leq \varepsilon\|Mf_k\|_{L^2(\mathbb{R}^n)} \lesssim \varepsilon \|f_k\|_{L^2(\mathbb{R}^n)}\lesssim \varepsilon.
\end{align*}
Thus $\lim_{k\rightarrow\infty}\|P_{\beta}f_k\|_{L^2(\mathbb{R}^n)} = 0$, completing the proof.
\end{proof}

\begin{proof}[Proof of Theorem \ref{CZOCompactness}]
The necessity of $T1, T^*1\in\text{CMO}(\mathbb{R}^n)$ follows from the fact that compact $T: L^2(\mathbb{R}^n)\to L^2(\mathbb{R}^n)$ extends to a compact operator $T: L^\infty(\mathbb{R}^n)\to \text{BMO}(\mathbb{R}^n)$ and the fact that each bounded Calder\'on-Zygmund operator takes $C_0(\R^n)$ functions into $\text{CMO}(\mathbb{R}^n)$. The  weak compactness condition is necessary by the argument following Definition \ref{weakcompdef}. 

We now prove sufficiency. Write
$$
    T=S+P_{T1}+P_{T^*1}^*.
$$
Then $S$ is a weakly compact Calder\'on-Zygmund operator satisfying $S1=S^*1=0$, so $S$ is localized by Theorem \ref{CZOlocalization} and therefore compact on $L^2(\mathbb{R}^n)$ by Theorem \ref{Generalcompactness}. The paraproducts $P_{T1}$ and $P_{T^*1}^*$ are both compact by Theorem \ref{ParaproductCompactness} since $T1,T^*1 \in \text{CMO}(\mathbb{R}^n)$.
\end{proof}

\end{document}